\documentclass[11pt]{amsart}

\usepackage{amsmath,amssymb,mathtools,cite}
\usepackage[hidelinks]{hyperref}

\allowdisplaybreaks

\newtheorem{theorem}{Theorem}[section]
\newtheorem{proposition}[theorem]{Proposition}
\newtheorem{lemma}[theorem]{Lemma}
\newtheorem{corollary}[theorem]{Corollary}

\newcommand{\C}{\mathbb C}

\newcommand{\Hol}{\operatorname{Hol}}

\begin{document}

\title[Products of Berezin Transforms]
{Products of Berezin Transforms on Generalized Fock Spaces}

\author{Hao Tian}
\address{School of Mathematics, East China University of
Science and Technology, Shanghai, 200237, China}
\email{tianhao20000420@163.com}
\thanks{Corresponding author: Hao Tian}


\begin{abstract}
We study products $B_{\alpha,m}^{(d)}B_{\beta,m}^{(d)}$ of Berezin transforms on generalized Fock spaces over $\mathbb C^d$. We establish a result for $m\neq2$: if $B_{\alpha,m}^{(d)}B_{\beta,m}^{(d)}$ coincides with another such product, then the weight parameters of the second product must be $(\alpha,\beta)$. As consequences, we characterize when two Berezin transforms commute and when the product of two Berezin transforms is again a Berezin transform.
\end{abstract}

\maketitle

\section{Introduction}

In the classical Fock space,
the Berezin transform coincides with the
Euclidean heat transform,
which is given by convolution
with a Gaussian kernel.
This connection has long been used to study
Toeplitz operators, quantization, and approximation on Fock spaces
\cite{BergerCoburn,Bauer2009,BauerCoburnIsralowitz,
BauerCoburn,BauerCoburnHagger}.
More recently, the Berezin transform has been used to study boundedness and compactness of Toeplitz operators through quantum harmonic analysis \cite{DewageMitkovski2025}, and to study Toeplitz operators on polyanalytic Fock spaces \cite{FulscheHagger2024}.
Fix $d\geq1$ and $m>0$, and put $q=\frac{2}{m}$.
For $\alpha>0$, let
\begin{equation}\label{eq:measure}
 \mathrm{d}\mu_{\alpha,m}^{(d)}(z)
 =
 \frac{\Gamma(d)}{\pi^d q\Gamma(qd)}
 \alpha^{qd}e^{-\alpha|z|^m}\,
 \mathrm{d}V(z)
\end{equation}
be the normalized measure on $\C^d$, where $\mathrm{d}V$ is the volume measure.
Denote by $F_{\alpha,m}^2(\C^d)$
the generalized Fock space
of entire functions on $\C^d$ that
are square integrable
with respect to
the measure
$\mu_{\alpha,m}^{(d)}$;
that is,
$$
 F_{\alpha,m}^2(\C^d)
 =
L^2(\C^d,\mu_{\alpha,m}^{(d)})\cap\Hol(\C^d).
$$
Let
$K_{\alpha,m}^{(d)}$ denote the reproducing kernel of
$F_{\alpha,m}^2(\C^d)$. The Berezin transform is given by
\begin{equation}\label{eq:Berezin}
 (B_{\alpha,m}^{(d)}f)(z)
 =
 \int_{\C^d}f(w)
 \frac{|K_{\alpha,m}^{(d)}(w,z)|^2}
 {K_{\alpha,m}^{(d)}(z,z)}
 \,\mathrm{d}\mu_{\alpha,m}^{(d)}(w),
 \qquad f\in L^\infty(\C^d).
\end{equation}
For background on generalized Fock spaces and Berezin transforms, see
\cite{Berezin,BommierHato2010Berezin,BommierHatoYoussfi2007,
BommierHatoYoussfi2010,BommierHatoEnglisYoussfi2012,
Englis1994,Englis1995,Zhu}.
Since each $B_{\alpha,m}^{(d)}$ maps $L^\infty(\C^d)$ into $L^\infty(\C^d)$, their composition is well defined, and we define
$$
B_{\alpha,m}^{(d)}B_{\beta,m}^{(d)}f=B_{\alpha,m}^{(d)}(B_{\beta,m}^{(d)}f).
$$
We first consider the classical Gaussian case $m=2$,
where the Berezin transform reduces to the Euclidean heat transform.
Precisely, by \cite[Section~3, (3.1)]{BauerCoburn},
$B_{\alpha,2}^{(d)}=Q_{1/(4\alpha)}$, where
$$
(Q_tf)(w)=\frac{1}{(4\pi t)^d}
\int_{\C^d}f(w-z)e^{-\frac{|z|^2}{4t}}\,\mathrm{d}V(z),~
t>0,
$$
and $Q_tQ_s=Q_{t+s},~s,t>0$. Therefore, we obtain
\begin{equation*}
\begin{aligned}
Q_{\frac{1}{4\alpha}}
Q_{\frac{1}{4\beta}}=
Q_{\frac{1}{4\alpha}+\frac{1}{4\beta}}=
Q_{\frac{1}{4\beta}}Q_{\frac{1}{4\alpha}};
\end{aligned}
\end{equation*}
\begin{equation}\label{eq:Gaussian-composition}
\begin{aligned}
B_{\alpha,2}^{(d)}B_{\beta,2}^{(d)}
=
B_{1/\left(\frac{1}{\alpha}+\frac{1}{\beta}\right),2}^{(d)}
=B_{\beta,2}^{(d)}B_{\alpha,2}^{(d)}.
\end{aligned}
\end{equation}
Thus
$\{B_{\alpha,2}^{(d)}:\alpha>0\}$ is commutative and
closed under products.
This leads to the following question for $m\ne2$.

\noindent
\textbf{Question 1.}
\emph{For fixed $d\geq1$ and $m>0$ with $m\neq2$, when does
$$B_{\alpha,m}^{(d)}B_{\beta,m}^{(d)}=
B_{\eta,m}^{(d)}B_{\gamma,m}^{(d)}$$
hold for $\alpha,\beta,\eta,\gamma>0$ ?}

Question 1 fixes $m$ and asks when two products determined by ordered pairs $(\alpha,\beta)$ and $(\eta,\gamma)$ can coincide.
The following Question 2 fixes a pair $\alpha\neq\beta$ and asks for which $m$ the corresponding Berezin transforms commute.
In the one-variable case, Borichev, Fantolini, and Youssfi \cite{BFY}
recently proved that the family
$\{B_{\alpha,m}^{(1)}:\alpha>0\}$
is commutative if and only if $m=2$.
They also proved that, when $m$ is not an even positive integer,
every fixed pair of distinct parameters fails to commute.
This leads to the following question.

\noindent
\textbf{Question 2.}
\emph{For fixed $d\geq1$ and $\alpha,\beta>0$ with $\alpha\neq\beta$,
when does
$$B_{\alpha,m}^{(d)}B_{\beta,m}^{(d)}=
B_{\beta,m}^{(d)}B_{\alpha,m}^{(d)}$$~hold for $m>0$ ?}

The following theorem gives a complete answer to Question~1.

\begin{theorem}\label{thm:main}
Let $d\geq1$, $m>0$ with $m\ne2$, and let
$\alpha,\beta,\eta,\gamma>0$. Then
$$
B_{\alpha,m}^{(d)}B_{\beta,m}^{(d)}
=
B_{\eta,m}^{(d)}B_{\gamma,m}^{(d)}
\iff
(\alpha,\beta)=(\eta,\gamma).
$$
\end{theorem}

Thus, different parameter pairs may yield the same product when $m=2$, while the situation is different when $m\ne2$.
The proof of Theorem~\ref{thm:main} shows how $\alpha$ and $\beta$ can be obtained from the product $B_{\alpha,m}^{(d)}B_{\beta,m}^{(d)}$. Applying the operator identity to the test functions $f_\delta(z)=e^{-\delta|z|^m}$ and evaluating at the origin gives
$
\Phi_{\alpha,\beta}^{(d)}(\delta)
=
\Phi_{\eta,\gamma}^{(d)}(\delta),
~\delta>0$,
where $\Phi_{\alpha,\beta}^{(d)}$ is given by \eqref{eq:Phi-integral}. 
The difficulty is 
that the identity on $\delta>0$ does not directly determine $\alpha$ and $\beta$ separately.
We therefore extend $\Phi_{\alpha,\beta}^{(d)}$ real analytically to
$
\left(-\frac{\alpha\beta}{\alpha+\beta},\infty\right)
$
and study its behavior at the left endpoint; see Theorem~\ref{prop:critical-signature}.

With the help of \eqref{eq:Gaussian-composition},
Theorem~\ref{thm:main} provides a complete characterization
for all $m>0$.

\begin{corollary}\label{cor:classification}
Let $d\geq1$, $m>0$, and let $\alpha,\beta,\eta,\gamma>0$. Then
$$
B_{\alpha,m}^{(d)}B_{\beta,m}^{(d)}
=
B_{\eta,m}^{(d)}B_{\gamma,m}^{(d)}
\iff
\begin{cases}
\displaystyle
\frac{1}{\alpha}+\frac{1}{\beta}
=
\frac{1}{\eta}+\frac{1}{\gamma},
& m=2,\\
(\alpha,\beta)=(\eta,\gamma),
& m\ne2.
\end{cases}
$$
\end{corollary}

We then turn to the commutativity problem discussed above.
Corollary
\ref{cor:commutativity} 
covers all values of 
the exponent $m>0$ 
in the weight 
$e^{-\alpha|z|^m}$, 
including the remaining cases
$m=4,6,\ldots$ not covered 
in \cite{BFY}. 
It also extends the 
one-variable result to 
the several-variable setting and thus gives a complete answer to Question~2.

\begin{corollary}\label{cor:commutativity}
Let $d\geq1$, $m>0$, and $\alpha,\beta>0$ with $\alpha\neq\beta$. Then
$$
B_{\alpha,m}^{(d)}B_{\beta,m}^{(d)}
=
B_{\beta,m}^{(d)}B_{\alpha,m}^{(d)}
\quad\Longleftrightarrow\quad
m=2.
$$
\end{corollary}

Corollary~\ref{cor:closure}
shows that 
the product of two Berezin transforms 
is again a Berezin transform 
if and only if $m=2$.

\begin{corollary}\label{cor:closure}
Let $d\geq1$, $m>0$, and $\alpha,\beta>0$. There exists $\theta>0$ such that
$B_{\alpha,m}^{(d)}B_{\beta,m}^{(d)}=B_{\theta,m}^{(d)}$
if and only if $m=2$. Consequently,
$\{B_{\alpha,m}^{(d)}:\alpha>0\}$ is closed under products if and
only if $m=2$.
\end{corollary}

The rest of the paper is organized as follows.
Section~2 collects the preliminaries and
derives the integral representation for 
the product of 
two Berezin transforms.
Section~3 proves
Theorem \ref{thm:main} and its corollaries.

\section{Preliminaries}

In this section, we recall the reproducing kernel in terms of the Prabhakar function, compute the Berezin transform of the radial test functions $f_\delta(z)=e^{-\delta|z|^m}$, and derive an integral representation for the product of two Berezin transforms.

The Prabhakar function, also known as
the three-parameter Mittag--Leffler function, is defined by
$$
 E_{a,b}^{c}(z)
 =
 \sum_{k=0}^\infty
 \frac{(c)_k}{k!\Gamma(ak+b)}z^k,
 ~z\in\mathbb C,
 ~a>0,~b,c\in\C,$$
where
$$
(c)_0=1,\qquad
(c)_k=\prod_{j=0}^{k-1}(c+j),\quad k\geq1.
$$
The Prabhakar function was introduced in connection with a singular integral equation and is now widely used in fractional calculus; see \cite{Prabhakar,GarraGarrappa,GiustiEtAl}.
By \cite[formula~(17)]{BommierHatoEnglisYoussfi2012},
we obtain
\begin{equation}\label{eq:kernel}
 K_{\alpha,m}^{(d)}(z,w)
=\Gamma(qd)E_{q,qd}^{d}\bigl(\alpha^q\langle z,w\rangle\bigr).
\end{equation}
For later use, set
$F(x)=E_{q,qd}^{d}(x)$.
For $\delta>0$, define
$f_\delta(z)=e^{-\delta|z|^m}$.

\begin{proposition}\label{prop:test-transform}
For $\alpha,\delta>0$ and $z\in\C^d$,
\begin{equation}\label{eq:test-transform}
 (B_{\alpha,m}^{(d)}f_\delta)(z)
 =
 \left(\frac{\alpha}{\alpha+\delta}\right)^{qd}
 \frac{
 F\bigl(\alpha^{2q}(\alpha+\delta)^{-q}|z|^2\bigr)
 }{
 F(\alpha^q|z|^2)
 }.
\end{equation}
\end{proposition}

\begin{proof}
Put $A=\alpha+\delta$. By \eqref{eq:measure},
$$
e^{-\delta|w|^m}\,\mathrm{d}\mu_{\alpha,m}^{(d)}(w)
=
\left(\frac{\alpha}{A}\right)^{qd}
\mathrm{d}\mu_{A,m}^{(d)}(w).
$$
Moreover, by \eqref{eq:kernel},
$$
K_{\alpha,m}^{(d)}(w,z)
=
K_{A,m}^{(d)}
\left(w,\left(\frac{\alpha}{A}\right)^qz\right).
$$
Therefore, by \eqref{eq:Berezin} and the reproducing property,
$$
\begin{aligned}
(B_{\alpha,m}^{(d)}f_\delta)(z)
&=
\frac{1}{K_{\alpha,m}^{(d)}(z,z)}
\int_{\C^d}
e^{-\delta|w|^m}
|K_{\alpha,m}^{(d)}(w,z)|^2
\,\mathrm{d}\mu_{\alpha,m}^{(d)}(w)
\\
&=
\left(\frac{\alpha}{A}\right)^{qd}
\frac{1}{K_{\alpha,m}^{(d)}(z,z)}
\int_{\C^d}
\left|
K_{A,m}^{(d)}
\left(w,\left(\frac{\alpha}{A}\right)^qz\right)
\right|^2
\,\mathrm{d}\mu_{A,m}^{(d)}(w)
\\
&=
\left(\frac{\alpha}{A}\right)^{qd}
\frac{
K_{A,m}^{(d)}
\left(
\left(\frac{\alpha}{A}\right)^qz,
\left(\frac{\alpha}{A}\right)^qz
\right)
}{
K_{\alpha,m}^{(d)}(z,z)
}
\\
&=
\left(\frac{\alpha}{A}\right)^{qd}
\frac{
F\bigl(\alpha^{2q}A^{-q}|z|^2\bigr)
}{
F\bigl(\alpha^q|z|^2\bigr)
}.
\end{aligned}
$$
Substituting $A=\alpha+\delta$ yields \eqref{eq:test-transform}.
\end{proof}

Since $K_{\alpha,m}^{(d)}(w,0)=1$, the formula \eqref{eq:Berezin} gives
\begin{equation}\label{eq:origin}
 (B_{\alpha,m}^{(d)}f)(0)
 =
 \frac{\Gamma(d)}{\pi^d q\Gamma(qd)}
 \alpha^{qd}
 \int_{\C^d}f(w)e^{-\alpha|w|^m}\,\mathrm{d}V(w).
\end{equation}
For $\alpha,\beta>0$, define
$$
\Phi_{\alpha,\beta}^{(d)}(\delta)
=
\bigl(
B_{\alpha,m}^{(d)}
B_{\beta,m}^{(d)}f_\delta
\bigr)(0),
\quad \delta>0.
$$
Set $A_\delta=\beta+\delta$ and $c_\delta=\beta^{2q}A_\delta^{-q}$.
Then \eqref{eq:origin} and Proposition~\ref{prop:test-transform} give
\begin{equation}\label{eq:Phi-integral}
\Phi_{\alpha,\beta}^{(d)}(\delta)
=
\frac{\Gamma(d)}{\pi^d q\Gamma(qd)}
\alpha^{qd}
\left(\frac{\beta}{A_\delta}\right)^{qd}
\int_{\C^d}
e^{-\alpha|z|^m}
\frac{F(c_\delta|z|^2)}{F(\beta^q|z|^2)}
\,\mathrm{d}V(z).
\end{equation}

\section{Proofs of the main results}

In this section, we prove
Theorem~\ref{thm:main}
and its
corollaries.

\subsection{Real-analytic continuation and endpoint asymptotics}\hfill

The following theorem plays a key role in the proof of Theorem~\ref{thm:main}.
Recall that a function $g$ is real analytic on an open interval $I$
if for every $x_0\in I$, there exist $r>0$ and coefficients
$a_n\in\mathbb{R}$, $n\geq0$, such that
$(x_0-r,x_0+r)\subset I$ and the series
$\sum_{n=0}^{\infty}a_n(x-x_0)^n$
converges to $g(x)$ for $|x-x_0|<r$.

\begin{theorem}\label{prop:critical-signature}
Let $\delta_{\alpha,\beta}
=-\frac{\alpha\beta}{\alpha+\beta}$.
The function $\Phi_{\alpha,\beta}^{(d)}$ has a real-analytic extension to
$(\delta_{\alpha,\beta},\infty)$. Moreover,
as
$\delta\to\delta_{\alpha,\beta}^{+}$,
\begin{equation}\label{eq:Phi-asymptotic}
 \Phi_{\alpha,\beta}^{(d)}(\delta)
 \sim
 L_{\alpha,\beta}^{(d)}
 (\delta-\delta_{\alpha,\beta})^{-qd},
\end{equation}
where
$L_{\alpha,\beta}^{(d)}
 =
 \beta^{d(2q-1)}\alpha^{qd}
 (\alpha+\beta)^{d(1-2q)}$.
\end{theorem}
\noindent
The proof requires three lemmas.
The first provides the asymptotic estimates for the Prabhakar function.
The second gives the analytic continuation needed to reach the critical endpoint,
while the third gives the precise asymptotic behavior at the critical endpoint.

The asymptotic expansions in
\cite[Theorems~3--5]{GarraGarrappa}
yield the following estimates.

\begin{lemma}\label{lem:F-asymptotic}
Set
$\sigma=\frac{d(1-q)}{q}$.
Then
\begin{align}\label{s3e1}
F(x)
=
\frac{1}{\Gamma(d)q^d}
x^\sigma e^{x^{1/q}}(1+o(1))
\end{align}
as $x\to+\infty$.
Moreover, if $J=[\rho_0,\rho_1]$ is a compact subinterval of $(0,\infty)$, then
\begin{align}\label{s3e2}
F(cr^2)
=
\frac{1}{\Gamma(d)q^d}
c^\sigma r^{2\sigma}e^{c^{1/q}r^m}
(1+o(1))
\end{align}
as $r\to\infty$, uniformly for $c\in J$.
\end{lemma}

\begin{proof}
Applying the asymptotic expansions in
\cite[Theorems~3--5]{GarraGarrappa}
with $(a,b,c)=(q,qd,d)$ and restricting to the positive real axis, we obtain
$$
F(x)\sim
\begin{cases}
\displaystyle
\frac{x^\sigma e^{x^{1/q}}}{\Gamma(d)q^d}
\sum_{k=0}^{\infty}c_kx^{-k/q}
+
\frac{(-1)^d x^{-d}}{\Gamma(d)}
\sum_{k=0}^{\infty}
\frac{\Gamma(k+d)}
{k!\Gamma(-qk)}
x^{-k},
& 0<q<2,
\\
\displaystyle
\frac{x^{-d/2}}{\Gamma(d)2^d}
\left(
e^{x^{1/2}}
\sum_{k=0}^{\infty}c_kx^{-k/2}
+
(-1)^d e^{-x^{1/2}}
\sum_{k=0}^{\infty}(-1)^kc_kx^{-k/2}
\right),
& q=2,
\\
\displaystyle
\frac{x^\sigma}{\Gamma(d)q^d}
\sum_{\ell=-P}^{P}
e^{2\pi i \ell\sigma}
\exp\left(
x^{1/q}e^{2\pi i \ell/q}
\right)
\sum_{k=0}^{\infty}
c_ke^{-2\pi i \ell k/q}x^{-k/q},
& q>2,
\end{cases}
$$
as $x\to+\infty$, where $c_0=1$, and when $q>2$,
$2P+1$ is the smallest odd integer greater than $q/2$.
To obtain \eqref{s3e1}, it suffices to show,
in each of the three cases,
that
$$
\frac{F(x)}
{\dfrac{1}{\Gamma(d)q^d}x^\sigma e^{x^{1/q}}}
\longrightarrow 1
~\text{as } x\to+\infty.
$$
\textbf{Case 1: $0<q<2$.} Since $c_0=1$, the first case of the expansion gives
$$
\begin{aligned}
\frac{F(x)}
{\dfrac{1}{\Gamma(d)q^d}x^\sigma e^{x^{1/q}}}
&=
1+O\left(x^{-1/q}\right)
+
O\left(
x^{-d-\sigma}e^{-x^{1/q}}
\right)
\to1~\mbox{as}~x\to+\infty.
\end{aligned}
$$

\noindent
\textbf{Case 2: $q=2$.} Then $\sigma=-d/2$, and the second case gives
$$
\begin{aligned}
&\frac{F(x)}
{\dfrac{1}{\Gamma(d)q^d}x^\sigma e^{x^{1/q}}}
=
\frac{F(x)}
{\dfrac{1}{\Gamma(d)2^d}x^{-d/2}e^{x^{1/2}}}
\\
=&
1+O\left(x^{-1/2}\right)
+
(-1)^d e^{-2x^{1/2}}
\left(1+O\left(x^{-1/2}\right)\right)
\to1~\mbox{as}~x\to+\infty.
\end{aligned}
$$

\noindent
\textbf{Case 3: $q>2$.} The term corresponding to $\ell=0$ satisfies
$$
\frac{
\frac{x^\sigma e^{x^{1/q}}}{\Gamma(d)q^d}
\sum\limits_{k=0}^{\infty}c_kx^{-k/q}
}{
\frac{1}{\Gamma(d)q^d}x^\sigma e^{x^{1/q}}
}
=
1+O\left(x^{-1/q}\right)
\to1.
$$
By the definition of $P$,
$2P-1\leq\frac q2<2P+1$.
Since $q>2$, it follows that $P<q/2$.
Hence, if $0<|\ell|\leq P$, then
$0<\left|\frac{2\pi \ell}{q}\right|<\pi$,
and therefore
$\cos\left(\frac{2\pi \ell}{q}\right)<1$.
For every such $\ell$, using $c_0=1$, we have
$$
\begin{aligned}
&
\left|
\frac{\frac{x^\sigma}{\Gamma(d)q^d}
e^{2\pi i \ell\sigma}
\exp\left(
x^{1/q}e^{2\pi i \ell/q}
\right)
\sum\limits_{k=0}^{\infty}
c_ke^{-2\pi i \ell k/q}x^{-k/q}
}{
\frac{1}{\Gamma(d)q^d}x^\sigma e^{x^{1/q}}
}
\right|
\\
=&
\exp\left[
x^{1/q}
\left(
\cos\frac{2\pi \ell}{q}-1
\right)
\right]
\left(
1+O\left(x^{-1/q}\right)
\right)
\to0.
\end{aligned}
$$
Since there are only finitely many indices $\ell$, the third case gives
$$
\frac{F(x)}
{\dfrac{1}{\Gamma(d)q^d}x^\sigma e^{x^{1/q}}}
\to1~\mbox{as}~x\to+\infty.
$$

Thus, in all three cases,
$$
F(x)
=\frac{1}{\Gamma(d)q^d}
x^\sigma e^{x^{1/q}}(1+o(1)),~x\to+\infty.
$$
This proves \eqref{s3e1}.

It remains to prove \eqref{s3e2}.
Write
$$
F(x)
=
\frac{1}{\Gamma(d)q^d}
x^\sigma e^{x^{1/q}}
\bigl(1+\varepsilon(x)\bigr),
$$
where $\varepsilon(x)\to0$ as $x\to+\infty$.
Let $J=[\rho_0,\rho_1]\subset(0,\infty)$. For $\rho\in J$,
$$
F(\rho r^2)
=
\frac{1}{\Gamma(d)q^d}
\rho^\sigma r^{2\sigma}
e^{\rho^{1/q}r^m}
\bigl(1+\varepsilon(\rho r^2)\bigr).
$$
Since $\rho r^2\geq\rho_0r^2$ for $\rho\in J$, we have
$$
\sup_{\rho\in J}
\left|\varepsilon(\rho r^2)\right|
\leq
\sup_{x\geq\rho_0r^2}
|\varepsilon(x)|
\longrightarrow0.
$$
Therefore,
$$
F(\rho r^2)
=
\frac{1}{\Gamma(d)q^d}
\rho^\sigma r^{2\sigma}
e^{\rho^{1/q}r^m}
(1+o(1))
$$
as $r\to\infty$, uniformly for $\rho\in J$.
\end{proof}

Fix $\alpha,\beta>0$ and put
$c_*=(\alpha+\beta)^q$.
For $|\lambda|<c_*$, define
$$
 H_{\alpha,\beta}^{(d)}(\lambda)
 =
 \int_{\C^d}
 e^{-\alpha|z|^m}
 \frac{F(\lambda|z|^2)}{F(a|z|^2)}
 \,\mathrm{d}V(z),
 ~\mbox{where}~a=\beta^q.
$$

\begin{lemma}\label{lem:H-holomorphic}
The function $H_{\alpha,\beta}^{(d)}$ is holomorphic on the disk $|\lambda|<c_*$.
\end{lemma}

\begin{proof}
For $\lambda\in\C$ and $z\in\C^d$, set
$$
G(\lambda,z)
=
e^{-\alpha|z|^m}
\frac{F(\lambda|z|^2)}{F(a|z|^2)}.
$$
Fix $0<\rho<c_*$. Since $F$ has positive Taylor coefficients,
$F(a|z|^2)>0$ and
$$
|F(\lambda|z|^2)|
\leq
F(\rho|z|^2),
\qquad |\lambda|\leq\rho.
$$
By Lemma~\ref{lem:F-asymptotic}, there exist $C,R>0$ such that
$$
|G(\lambda,z)|
\leq
C e^{-(\alpha+\beta-\rho^{1/q})|z|^m},
\qquad
|\lambda|\leq\rho,\quad |z|\geq R.
$$
Since $\rho<c_*\equiv (\alpha+\beta)^q$, the function on the right is
integrable over $\C^d$.
On the compact set
$\{|\lambda|\leq\rho,\ |z|\leq R\}$,
the function $G$ is bounded. Combining this with the estimate for $|z|\geq R$, we obtain
an integrable function $h_\rho\in L^1(\C^d)$ such that
$$
|G(\lambda,z)|\leq h_\rho(z),
\qquad |\lambda|\leq\rho.
$$
Hence, by dominated convergence,
$H_{\alpha,\beta}^{(d)}$ is continuous on $|\lambda|<c_*$. Moreover, if
$\Delta$ is a triangle compactly contained in this disk, then Fubini's
theorem and Cauchy's theorem give
$$
\int_{\partial\Delta}
H_{\alpha,\beta}^{(d)}(\lambda)\,\mathrm{d}\lambda
=
\int_{\C^d}
\left(
\int_{\partial\Delta}G(\lambda,z)\,\mathrm{d}\lambda
\right)\,\mathrm{d}V(z)
=0.
$$
Morera's theorem therefore implies that
$H_{\alpha,\beta}^{(d)}$ is holomorphic on $|\lambda|<c_*$.
\end{proof}

\begin{lemma}\label{lem:H-asymptotic}
As $c\to c_*^{-}$,
$$
 H_{\alpha,\beta}^{(d)}(c)
 \sim
 \left(\frac{c}{a}\right)^\sigma
 \frac{\pi^d q\Gamma(qd)}{\Gamma(d)}
 \bigl(c_*^{1/q}-c^{1/q}\bigr)^{-qd}.
$$
\end{lemma}

\begin{proof}
Set
$$
\kappa=c_*^{1/q}-c^{1/q}
=\alpha+\beta-c^{1/q},
$$
so that $\kappa\to0^+$ as $c\to c_*^-$. Choose $\varepsilon_0>0$ such that
$
J=[c_*-\varepsilon_0,c_*]\subset(0,\infty)$.

We first establish
$$
e^{-\alpha r^m}\frac{F(cr^2)}{F(ar^2)}
\sim
\left(\frac ca\right)^\sigma e^{-\kappa r^m}
\qquad (r\to\infty),
$$
uniformly for $c\in J$. By Lemma~\ref{lem:F-asymptotic},
$$
F(cr^2)
=
\frac{1}{\Gamma(d)q^d}
c^\sigma r^{2\sigma}e^{c^{1/q}r^m}
(1+u_c(r)),
$$
where
$
\sup_{c\in J}|u_c(r)|\to0
~(r\to\infty)$,
and
$$
F(ar^2)
=
\frac{1}{\Gamma(d)q^d}
a^\sigma r^{2\sigma}e^{\beta r^m}
(1+v(r)),
\qquad
v(r)\to0.
$$
Hence
$$
e^{-\alpha r^m}\frac{F(cr^2)}{F(ar^2)}
=
\left(\frac ca\right)^\sigma
e^{-\kappa r^m}
\frac{1+u_c(r)}{1+v(r)},
$$
with
$
\sup_{c\in J}
\left|
\frac{1+u_c(r)}{1+v(r)}-1
\right|
\to0$.
Thus, for every $\varepsilon\in(0,1)$, there exists
$R=R_\varepsilon>0$, independent of $c\in J$, such that
\begin{equation}\label{eq:H-integrand-comparison}
(1-\varepsilon)
\left(\frac ca\right)^\sigma e^{-\kappa r^m}
\leq
e^{-\alpha r^m}\frac{F(cr^2)}{F(ar^2)}
\leq
(1+\varepsilon)
\left(\frac ca\right)^\sigma e^{-\kappa r^m}
\end{equation}
for $r\geq R$ and $c\in J\cap(0,c_*)$.

We next show that
$$
H_{\alpha,\beta}^{(d)}(c)
\sim
\left(\frac ca\right)^\sigma
\int_{\C^d}e^{-\kappa|z|^m}\,\mathrm{d}V(z).
$$
Write
$
H_{\alpha,\beta}^{(d)}(c)=I_R(c)+T_R(c)$,
where
$$
I_R(c)
=
\int_{|z|\leq R}
e^{-\alpha|z|^m}
\frac{F(c|z|^2)}{F(a|z|^2)}
\,\mathrm{d}V(z),~
T_R(c)
=
\int_{|z|\geq R}
e^{-\alpha|z|^m}
\frac{F(c|z|^2)}{F(a|z|^2)}
\,\mathrm{d}V(z).
$$
Since the integrand is continuous on
$
J\times\{z\in\C^d:|z|\leq R\}$,
there exists $C_R>0$ such that
$
0\leq I_R(c)\leq C_R,
~c\in J\cap(0,c_*)$.
By \eqref{eq:H-integrand-comparison},
\begin{align}
(1-\varepsilon)
\left(\frac ca\right)^\sigma
\int_{|z|\geq R}e^{-\kappa|z|^m}\,\mathrm{d}V(z)
\notag
\leq T_R(c)\leq
(1+\varepsilon)
\left(\frac ca\right)^\sigma
\int_{|z|\geq R}e^{-\kappa|z|^m}\,\mathrm{d}V(z).
\label{eq:H-tail-comparison}
\end{align}
By polar coordinates,
$$
\begin{aligned}
\int_{\C^d}e^{-\kappa|z|^m}\,\mathrm{d}V(z)
&=
\frac{2\pi^d}{\Gamma(d)}
\int_0^\infty e^{-\kappa r^m}r^{2d-1}\,\mathrm{d}r=
\frac{\pi^d q\Gamma(qd)}{\Gamma(d)}
\kappa^{-qd}.
\end{aligned}
$$
In particular,
$
\int_{\C^d}e^{-\kappa|z|^m}\,\mathrm{d}V(z)\to\infty
~(\kappa\to0^+)$.
Moreover,
$$
0\leq
\int_{\C^d}e^{-\kappa|z|^m}\,\mathrm{d}V(z)
-
\int_{|z|\geq R}e^{-\kappa|z|^m}\,\mathrm{d}V(z)
\leq \frac{\pi^dR^{2d}}{\Gamma(d+1)},
$$
and therefore
$$
\frac{
\displaystyle\int_{|z|\geq R}e^{-\kappa|z|^m}\,\mathrm{d}V(z)
}{
\displaystyle\int_{\C^d}e^{-\kappa|z|^m}\,\mathrm{d}V(z)
}
\to1.
$$
Thus,
$$
1-\varepsilon
\leq
\liminf_{c\to c_*^-}
\frac{
H_{\alpha,\beta}^{(d)}(c)
}{
\displaystyle
\left(\frac ca\right)^\sigma
\int_{\C^d}e^{-\kappa|z|^m}\,\mathrm{d}V(z)
}
\leq
\limsup_{c\to c_*^-}
\frac{
H_{\alpha,\beta}^{(d)}(c)
}{
\displaystyle
\left(\frac ca\right)^\sigma
\int_{\C^d}e^{-\kappa|z|^m}\,\mathrm{d}V(z)
}
\leq
1+\varepsilon.
$$
Letting $\varepsilon\to0^+$ yields
$$
H_{\alpha,\beta}^{(d)}(c)
\sim
\left(\frac ca\right)^\sigma
\int_{\C^d}e^{-\kappa|z|^m}\,\mathrm{d}V(z)
=
\left(\frac ca\right)^\sigma
\frac{\pi^d q\Gamma(qd)}{\Gamma(d)}
\bigl(c_*^{1/q}-c^{1/q}\bigr)^{-qd}.
$$
This proves the desired asymptotic formula.
\end{proof}

We are ready to prove Theorem~\ref{prop:critical-signature}.

\begin{proof}[Proof of Theorem~\ref{prop:critical-signature}]
For $\delta>\delta_{\alpha,\beta}$, we have
$c_\delta<c_*$.
Moreover, since
$$
\beta+\delta
>
\beta+\delta_{\alpha,\beta}=
\frac{\beta^2}{\alpha+\beta}
>
0,
$$
the function
$\delta\longmapsto c_\delta=
\beta^{2q}(\beta+\delta)^{-q}$
is real analytic on
$(\delta_{\alpha,\beta},\infty)$ and takes values in $(0,c_*)$.
By Lemma~\ref{lem:H-holomorphic},
$H_{\alpha,\beta}^{(d)}$ is holomorphic on $|\lambda|<c_*$, and hence its restriction to $(0,c_*)$ is real analytic. Therefore,
$\delta\longmapsto
H_{\alpha,\beta}^{(d)}(c_\delta)$
is real analytic on
$(\delta_{\alpha,\beta},\infty)$.
The remaining factors in \eqref{eq:Phi-integral} are real analytic on the same interval. Therefore, \eqref{eq:Phi-integral} defines the required continuation.

Set
$$
 \kappa(\delta)
 =
 c_*^{1/q}-c_\delta^{1/q}
 =
 \alpha+\beta-
 \frac{\beta^2}{\beta+\delta}.
$$
By Lemma~\ref{lem:H-asymptotic} and \eqref{eq:Phi-integral},
as
$\delta\to\delta_{\alpha,\beta}^{+}$,
$$
 \Phi_{\alpha,\beta}^{(d)}(\delta)
 \sim
 \alpha^{qd}
 \left(\frac{\beta}{A_\delta}\right)^{qd}
 \left(\frac{c_\delta}{a}\right)^\sigma
 \kappa(\delta)^{-qd}.
$$
Since
$
 \frac{c_\delta}{a}
 =
 \left(\frac{\beta}{A_\delta}\right)^q
$
and $q\sigma=d(1-q)$, the product of the two middle factors is
$
 \left(\frac{\beta}{A_\delta}\right)^d
$.
Furthermore,
$
\kappa(\delta)
=
\frac{(\alpha+\beta)
(\delta-\delta_{\alpha,\beta})}
{\beta+\delta}$.
It follows that
$$
 \Phi_{\alpha,\beta}^{(d)}(\delta)
 \sim
 \beta^d\alpha^{qd}
 (\beta+\delta)^{d(q-1)}
 (\alpha+\beta)^{-qd}
 (\delta-\delta_{\alpha,\beta})^{-qd}.
$$
Finally,
$
\beta+\delta_{\alpha,\beta}
=
\frac{\beta^2}{\alpha+\beta}$,
which gives the stated value of $L_{\alpha,\beta}^{(d)}$.
\end{proof}

\subsection{Proof of Theorem~\ref{thm:main}}\hfill

\begin{proof}[Proof of Theorem~\ref{thm:main}]
The sufficiency is immediate. For the necessity, assume that
$$
B_{\alpha,m}^{(d)}B_{\beta,m}^{(d)}
=
B_{\eta,m}^{(d)}B_{\gamma,m}^{(d)}.
$$
Applying both sides to $f_\delta$ and evaluating at the origin gives
$$
\Phi_{\alpha,\beta}^{(d)}(\delta)
=
\Phi_{\eta,\gamma}^{(d)}(\delta),
\qquad \delta>0.
$$
Set
$$
\delta_1=-\frac{\alpha\beta}{\alpha+\beta},
\qquad
\delta_2=-\frac{\eta\gamma}{\eta+\gamma},
\qquad
\delta_0=\max\{\delta_1,\delta_2\}.
$$
By Theorem~\ref{prop:critical-signature},
$\Phi_{\alpha,\beta}^{(d)}$
and
$\Phi_{\eta,\gamma}^{(d)}$
are real analytic on
$(\delta_1,\infty)$ and $(\delta_2,\infty)$, respectively. Since they agree on
$(0,\infty)$, we have
\begin{equation}\label{eq:Phi-identity}
\Phi_{\alpha,\beta}^{(d)}(\delta)
=
\Phi_{\eta,\gamma}^{(d)}(\delta),
\qquad
\delta\in(\delta_0,\infty).
\end{equation}

We show that $\delta_1=\delta_2$. Otherwise, after interchanging the two
pairs if necessary, suppose that $\delta_1>\delta_2$. Then
Theorem~\ref{prop:critical-signature} gives
$$
\Phi_{\alpha,\beta}^{(d)}(\delta)\to+\infty
\qquad (\delta\to\delta_1^+),
$$
whereas $\delta_1\in(\delta_2,\infty)$, so
$\Phi_{\eta,\gamma}^{(d)}$ is bounded in a
neighborhood of $\delta_1$. This contradicts \eqref{eq:Phi-identity}. Therefore,
$\delta_1=\delta_2$.

Put
$
s:=
\frac{\alpha\beta}{\alpha+\beta}
=
\frac{\eta\gamma}{\eta+\gamma}$.
Then \eqref{eq:Phi-identity} holds for $\delta>-s$.
Using Theorem~\ref{prop:critical-signature},
we obtain
$$
L_{\alpha,\beta}^{(d)}
=
\lim_{\delta\to-s^+}
(\delta+s)^{qd}\Phi_{\alpha,\beta}^{(d)}(\delta)
=
\lim_{\delta\to-s^+}
(\delta+s)^{qd}\Phi_{\eta,\gamma}^{(d)}(\delta)
=
L_{\eta,\gamma}^{(d)}.
$$
Set
$
r=\frac{\alpha}{\beta},~
t=\frac{\eta}{\gamma}$.
Since
$
s=\frac{\alpha\beta}{\alpha+\beta}
=\frac{\eta\gamma}{\eta+\gamma}$,
we have
$$
\alpha=s(1+r),
\qquad
\beta=\frac{s(1+r)}{r},
$$
and similarly
$$
\eta=s(1+t),
\qquad
\gamma=\frac{s(1+t)}{t}.
$$
Substitution into the formula for $L_{\alpha,\beta}^{(d)}$ gives
\begin{equation}\label{eq:L-s-r}
L_{\alpha,\beta}^{(d)}
=
s^{qd}(1+r)^{d(1-q)},
\end{equation}
and likewise
$$
L_{\eta,\gamma}^{(d)}
=
s^{qd}(1+t)^{d(1-q)}.
$$
Hence
$
(1+r)^{d(1-q)}
=
(1+t)^{d(1-q)}$.
Since $m\ne2$, we have $q\ne1$, and therefore
$r\mapsto(1+r)^{d(1-q)}$ is injective on $(0,\infty)$. Thus $r=t$, and consequently
$$
\alpha=s(1+r)=\eta,
\qquad
\beta=\frac{s(1+r)}{r}=\gamma.
$$
Hence $(\alpha,\beta)=(\eta,\gamma)$.
\end{proof}

\begin{proof}[Proof of Corollary~\ref{cor:classification}]
For $m\ne2$, the result follows from Theorem~\ref{thm:main}. Let $m=2$ and set
$
s=\frac{\alpha\beta}{\alpha+\beta}$ and
$t=\frac{\eta\gamma}{\eta+\gamma}$.
By \eqref{eq:Gaussian-composition},
$$
B_{\alpha,2}^{(d)}B_{\beta,2}^{(d)}
=Q_{\frac{\alpha+\beta}{4\alpha\beta}}=
B_{s,2}^{(d)},
\qquad
B_{\eta,2}^{(d)}B_{\gamma,2}^{(d)}
=Q_{\frac{\eta+\gamma}{4\eta\gamma}}=
B_{t,2}^{(d)}.
$$
Thus it remains to show that
$
B_{s,2}^{(d)}=B_{t,2}^{(d)}
\Longleftrightarrow
s=t$.
By \eqref{eq:test-transform} with $m=2$ and $z=0$,
the equality
$
B_{s,2}^{(d)}=B_{t,2}^{(d)}
$
implies
$$
\left(\frac{s}{s+\delta}\right)^d
=
\left(\frac{t}{t+\delta}\right)^d,
\qquad \delta>0,
$$
and hence $s=t$.
\end{proof}

\begin{proof}[Proof of Corollary~\ref{cor:commutativity}]
This follows immediately from Corollary~\ref{cor:classification} and the assumption $\alpha\ne\beta$.
\end{proof}

\begin{proof}[Proof of Corollary~\ref{cor:closure}]
For $m=2$, the result follows from \eqref{eq:Gaussian-composition}. Suppose
$m\ne2$ and
$$
B_{\alpha,m}^{(d)}B_{\beta,m}^{(d)}
=
B_{\theta,m}^{(d)}
$$
for some $\theta>0$. By \eqref{eq:test-transform} with $z=0$,
$$\Phi_{\alpha,\beta}^{(d)}(\delta)
=\bigl(
B_{\alpha,m}^{(d)}
B_{\beta,m}^{(d)}f_\delta
\bigr)(0)=\bigl(B_{\theta,m}^{(d)}
f_\delta
\bigr)(0)=
\left(\frac{\theta}{\theta+\delta}\right)^{qd},~\delta>0.$$
Set
$s=\frac{\alpha\beta}{\alpha+\beta}$.
By Theorem~\ref{prop:critical-signature},
$\Phi_{\alpha,\beta}^{(d)}$ is real analytic on $(-s,\infty)$, while
$
\delta\longmapsto
\left(\frac{\theta}{\theta+\delta}\right)^{qd}
$
is real analytic on $(-\theta,\infty)$. Hence,
$$
\Phi_{\alpha,\beta}^{(d)}(\delta)
=
\left(\frac{\theta}{\theta+\delta}\right)^{qd},
\qquad
\delta>\max\{-s,-\theta\}.
$$
We claim that $\theta=s$. If $s<\theta$, then $-s>-\theta$, and
$
\Phi_{\alpha,\beta}^{(d)}(\delta)\to+\infty
~(\delta\to-s^+)$
by Theorem~\ref{prop:critical-signature},
whereas
$
\left(\frac{\theta}{\theta+\delta}\right)^{qd}
\longrightarrow
\left(\frac{\theta}{\theta-s}\right)^{qd}<\infty$.
If $s>\theta$, then
$
\left(\frac{\theta}{\theta+\delta}\right)^{qd}\to+\infty
~(\delta\to-\theta^+)$,
whereas $\Phi_{\alpha,\beta}^{(d)}$ is real analytic at $-\theta$. Thus
$s=\theta$.
Therefore,
$$
\Phi_{\alpha,\beta}^{(d)}(\delta)
=
\left(\frac{s}{s+\delta}\right)^{qd},
\qquad \delta>-s.
$$
Multiplying by $(\delta+s)^{qd}$ and letting $\delta\to-s^+$ gives
$$
L_{\alpha,\beta}^{(d)}
=
\lim_{\delta\to-s^+}
(\delta+s)^{qd}\Phi_{\alpha,\beta}^{(d)}(\delta)
=
s^{qd}.
$$
Let $r=\alpha/\beta$. Then \eqref{eq:L-s-r} gives
$$
L_{\alpha,\beta}^{(d)}
=
s^{qd}(1+r)^{d(1-q)}.
$$
Hence
$(1+r)^{d(1-q)}=1$.
Since $m\ne2$, we have $q\ne1$, while $1+r>1$. Thus no such
$\theta$ exists when $m\ne2$.
\end{proof}

\noindent\textbf{Acknowledgment}
This work is partially supported by NSFC(12471122; 12271090).

\end{document}